\documentclass[a4paper,10pt]{scrartcl}
\usepackage{geometry}
\usepackage[utf8]{inputenc}
\usepackage[T1]{fontenc}
\usepackage{graphicx}
\usepackage{amsmath, amssymb, amstext}
\usepackage{amsthm}
\usepackage{enumitem}
\usepackage{tikz}
\usepackage{hhline, braket}
\usepackage{mathtools}
\pgfdeclarelayer{edgelayer}
\pgfdeclarelayer{nodelayer}
\pgfsetlayers{edgelayer,nodelayer,main}
\tikzstyle{rn}=[circle,fill=white,draw=black,line width=0.4 pt]

\usepackage{pgfplots}
\usepackage{float}
\usepackage[colorlinks]{hyperref}
\definecolor{lightblue}{rgb}{0.5,0.5,1.0}
\definecolor{darkred}{rgb}{0.5,0,0}
\definecolor{darkgreen}{rgb}{0,0.5,0}
\definecolor{darkblue}{rgb}{0,0,0.5}
\hypersetup{colorlinks,linkcolor=darkred,filecolor=darkgreen,urlcolor=darkred,citecolor=darkblue}
\usepackage[titlenumbered, linesnumbered, ruled]{algorithm2e}

\usetikzlibrary{patterns,arrows,decorations.pathreplacing,decorations.pathmorphing}

\theoremstyle{definition}
\newtheorem{theorem}{Theorem}[section]
\newtheorem*{theorem*}{Theorem}
\newtheorem*{maintheoremrestated}{Theorem~\ref{thm:characterization-two-factor-injective} (restated)}

\newtheorem{lemma}[theorem]{Lemma}
\newtheorem{definition}[theorem]{Definition}
\newtheorem{corollary}[theorem]{Corollary}
\newtheorem{example}[theorem]{Example}

\definecolor{darkred}{rgb}{0.5,0,0}

\numberwithin{equation}{section}
\numberwithin{figure}{section}

\newcommand{\leftmset}{\{\!\!\{}
\newcommand{\rightmset}{\}\!\!\}}

\title{Subgroups of 3-factor direct products}
\author{Daniel Neuen and Pascal Schweitzer\\ 
RWTH Aachen University\\
\texttt{\{neuen,schweitzer\}@informatik.rwth-aachen.de}
}

\begin{document}

\maketitle

\begin{abstract}
 Extending Goursat's Lemma we investigate the structure of subdirect products of 3-factor direct products. We give several example constructions and then provide a structure theorem showing that every such group is essentially obtained by a combination of the constructions. The central observation in this structure theorem is that the dependencies among the group elements in the subdirect product that involve all three factors are of Abelian nature. In the spirit of Goursat's Lemma, for two special cases, we derive correspondence theorems between data obtained from the subgroup lattices of the three factors (as well as isomorphism between arising factor groups) and the subdirect products.
 Using our results we derive an explicit formula to count the number of subdirect products of the direct product of three symmetric groups.
\end{abstract}
\section{Introduction}

The lemma of Goursat~\cite{MR1508819} is a classic result of group theory that characterizes the subgroups of a direct product of two groups~$G_1\times G_2$. A version of the lemma also provides means to describe the subgroups of~$G_1\times G_2$ by inspecting the subgroup lattices of~$G_1$ and~$G_2$ and considering isomorphisms between arising factor groups.

In an expository article, Anderson and Camillo~\cite{MR2508141} demonstrate for example the applicability of Goursat's lemma to determine normal subgroups of~$G_1\times G_2$, to count the number of subgroups of~$S_3\times S_3$, and to prove the Zassenhaus Lemma. They also describe how Goursat's Lemma can be stated in the context of rings, ideals, subrings and in modules. The Lemma itself can also be found in various introductory algebra and group theory texts (e.g.,~\cite{MR0414669,MR1878556}).

While Goursat's Lemma applies to subgroups of the direct product of two groups, in this work we are concerned with subgroups of the direct product of three groups.

It seems that there is no straightforward generalization to three factors. Indeed, Bauer, Sen, and Zvengrowski~\cite{GenGoursat} develop a generalization to an arbitrary finite number of factors by devising a non-symmetric version of Goursat's lemma for two factors that can then be applied recursively. A more category theory focused approach is taken by Gekas. However no simple correspondence theorem between the subdirect products of 3-factor direct products and data depending on the sublattice of the subgroups of the factors and isomorphisms between them is at hand. In fact, in~\cite{GenGoursat} the authors exhibit two Abelian examples that stand in the way of such a correspondence theorem by sharing the various characteristic subgroups and isomorphisms between them and yet being distinct. Both these papers are able to recover several identities provided by Remak~\cite{MR1581299} who is explicitly concerned with 3-factor subdirect products.

In this paper we analyze the structure of subdirect products of 3-factor direct products. To this end we give several example constructions of such groups and then provide a structure theorem showing that every such group is essentially obtained by a combination of the constructions. The central observation in this structure theorem is that the dependencies among the group elements in the subdirect product that involve all three factors are of Abelian nature.
We call a subdirect product of~$G_1\times G_2\times G_3$ 2-factor injective if each of the three projections onto two factors is injective. By taking suitable quotients, it is possible to restrict our investigations to 2-factor injective subdirect products, for which we obtain the following theorem.

\begin{theorem}[Characterization of 2-factor injective subdirect products of 3-factor products]
\label{thm:characterization-two-factor-injective}
Let $\Delta \leq G_1\times G_2\times G_3$ be a 2-factor injective subdirect product. Then there is a normal subgroup of~$H\trianglelefteq \Delta$ with~$[\pi_i(\Delta) \colon \pi_i(H)] = [\Delta \colon H]$ for~$i\in\{1,2,3\}$ and~$H$ is isomorphic to a group of the following form: there are three finite groups~$H_1,H_2,H_3$ that all have an Abelian subgroup~$M$ contained in their the center such that~$H$ is isomorphic to the factor group of 
\[\{((h_2,{h'_3}^{-1}),(h_3,{h'_1}^{-1}),(h_1,{h'_2}^{-1}))\mid  h_i,h'_i\in H_i, h_i {h'_i}^{-1} \in M, h_1 {h'_1}^{-1}h_2 {h'_2}^{-1}h_3 {h'_3}^{-1}=1\},\] 
by the normal subgroup~$\{((m_1,m_1),(m_2,m_2),(m_3,m_3))\mid m_i\in M\}$. 
\end{theorem}

In the spirit of Goursat's Lemma we then investigate correspondence theorems between data obtained from the subgroup lattices of the~$G_i$ (as well as isomorphism between arising factor groups) and the subdirect products of~$G_1\times G_2\times G_3$.
For two special cases, namely the cases~$H = \Delta$, and~$M = \{1\}$, we obtain such a correspondence theorem for three factors.
Here the second case is a particular special case hinted at in~\cite{GenGoursat}, which is indeed describable by a symmetric version of a generalized Goursat's Lemma.
In a third special case, where one of the~$G_i$ is the semi-direct product of the projection of~$H$ onto the~$i$-th component and some other group, we also obtain a partial correspondence theorem.

As demonstrated by Petrillo~\cite{MR2801828}, Goursat Lemma can readily be applied to count subgroups of the product of two Abelian groups. For a direct product of an arbitrary number of Abelian groups the number of subgroups has been extensively studied. We refer to the monograph of Butler~\cite{MR1223236}.
In fact there are also explicit formulas for the counts of subgroups of~$\mathbb{Z}_p\times \mathbb{Z}_q\times \mathbb{Z}_r$ (see for example~\cite{MR3045601}). 
In line with the papers and as an application of our characterization, we derive an explicit formula to count the number of subdirect products of the direct product of three symmetric groups~$S_{n_1}\times S_{n_2}\times S_{n_3}$. 
It is also possible for example to count the normal subgroups of such direct products. In fact, the normal subgroups can be also characterized for arbitrary finite products of symmetric groups~\cite{MR0399275}.
Let us finally also point to some literature concerning finiteness properties of groups~\cite{MR3086064,MR2500867} which also contains some structural results on 3-factor direct products (in particular on the case we call 3-factor surjective).

\section{Goursat's Lemma}

 Let $G = G_1 \times G_2 \times \dots \times G_t$ be a direct product of groups. 
 We define for~$i\in \{1,\ldots,t\}$ the map~$\pi_i$ as the projection to the~$i$-th coordinate and we define the homomorphism~$\psi_{i}: \Delta \rightarrow G_1  \times \dots \times G_{i-1}\times G_{i+1}\times \dots \times G_t: (g_1,g_2,\ldots,g_t) \mapsto (g_1,\ldots,g_{i-1},g_{i+1},\ldots,g_t)$.

 A subgroup~$\Delta\leq G$  of the direct product is said to be a \emph{subdirect product} if $\pi_i(\Delta) = G_i$ for all $1 \leq i \leq t$. Goursat's Lemma is a classic statement  concerned with the structure of subdirect products of direct products of two factors.

\begin{theorem}[Goursat's Lemma]
Let~$\Delta\leq G_1 \times G_2 =G$ be a subdirect product and define~$N_1 = \{g_1\in G_1\mid (g_1,1)\in \Delta \}$ as well as~$N_2 = \{g_2\in G_1\mid (1,g_2)\in \Delta \}$.
Then~$G_1/N_1$ is isomorphic to~$G/N_2$ via an isomorphism~$\varphi$ for which~$(g_1,g_2)\in \Delta$ if and only if~$\varphi(g_1)= g_2$. 
\end{theorem}

This gives a natural homomorphism~$\Delta \rightarrow G_1/N_1\times G_2/N_2$ defined as~$(g_1,g_2)\mapsto (g_1N_1,g_2N_2)$ with image~$\{(g_1,g_2) \mid \varphi(g_1) = g_2\}$. Thus we can view~$\Delta$ as a fiber product (or pull back) of~$G_1$ and~$G_2$ over~$G_i/N_i$.

A typical application of the lemma is a proof of the fact that subdirect products of non-Abelian finite simple groups are direct products of diagonal subgroups.
Furthermore, the lemma can be applied to count subgroups of direct products. For this, the following correspondence version of the lemma  is more convenient.

\begin{theorem}\label{thm:goursat:correspondence}
There is a natural one-to-one correspondence between the subgroups of~$G_1 \times G_2$ and the tuples~$(P_1,P_2,N_1,N_2,\varphi)$ for
which for~$i\in \{1,2\}$ we have 
\begin{enumerate}
\item $N_i\trianglelefteq P_i\leq G_i$ and
\item $P_1/N_1  \stackrel{\varphi}{\cong} P_2/N_2$.
\end{enumerate}
\end{theorem}

Here, we write~$G_1 \stackrel{\varphi}{\cong} G_2$ to denote that~$G_1$ and~$G_2$ are isomorphic via an isomorphism~$\varphi$.
The subdirect products correspond to those tuples for which~$P_1 = G_1$ and~$P_2  = G_2$. Diagonal subgroups are those subdirect products that also satisfy~$N_1 = N_2 =1$. Subproducts are those for which~$N_1 = P_1$ and~$N_2 = P_2$.

\section{Three factors}

We now focus on 3-factor subdirect products. Before we investigate the general case, we consider four examples of subdirect products. 

In our first examples, we consider groups that are~$2$-factor surjective. We say~$\Delta\leq G_1\times G_2\times G_3$ is
\emph{$2$-factor surjective} if~$\psi_{i}$ is surjective for all~$1 \leq i \leq 3$. Note that the analogous definition of~$1$-factor surjectivity (i.e., all~$\pi_i$ are surjective) means then the same as being subdirect.

Similarly, we say~$\Delta$ is \emph{$2$-factor injective} if~$\psi_{i}$ is injective for all~$1 \leq i \leq 3$. Note that this assumption is equivalent to saying that two components of an element of~$\Delta$ determine the third. Analogously~$1$-factor injective then means that one component determines the other two. 

\subsection{Examples of 3-factor direct products}

\begin{example}
The subgroup of~$(G_1)^3$ comprised of the set~$\{(g,g,g)\mid g\in G_1)\}$ is called the diagonal subgroup.
\end{example}

It is not difficult to see that the only~$1$-factor injective subdirect products are diagonal subgroups.
As a second example, let~$G_1$ be an Abelian group. Then the group
\begin{example}
 \label{ex:cfi-group}
 \begin{equation}
  \Delta \coloneqq \{(a,b,c) \in (G_1)^{3}\;|\;abc = 1\}.
 \end{equation}
 is a subdirect product of~$(G_1)^3$ that is~$2$-factor surjective and $2$-factor injective. 
 \end{example}
 It turns out that this is the only type of group with these properties. 
 
 \begin{lemma}
  Let $G = G_1\times G_2\times G_3$ be a group and~$\Delta$ a subdirect product of~$G$ that is~$2$-factor surjective and $2$-factor  injective. Then~$G_1$,~$G_2$ and~$G_3$ are isomorphic Abelian groups and~$\Delta$ is isomorphic to the subgroup of~$G_1^3$ given by~$\{(a,b,c) \in (G_1)^{3}\;|\;abc = 1\}$, which in turn is isomorphic to~$(G_1)^{2}$ as an abstract group.
\end{lemma}

\begin{proof}
Let~$g_1$ and~$g'_1$ be elements of~$G_1$. Then by~$2$-factor surjectivity there are elements~$g_2$ and~$g_3$ such that~$(g_1,g_2,1)\in \Delta$ and~$(g'_1,1,g_3)\in \Delta$.
We thus have that~$(g_1,g_2,1)^{(g'_1,1,g_3)}= (g_1^{g'_1},g_2,1)$. By 2-factor injectivity it follows that~$g_1^{g'_1} = g_1$ and thus~$G_1$ is Abelian.

For every~$g_1\in G_1$ there is exactly one element of the form~$(g_1,g_2,1)$ in~$\Delta$. The map~$\varphi$ that sends every~$g_1$ to the corresponding~$g_2$ provides us with a map from~$G_1$ to~$G_2$. By 2-factor surjectivity and 2-factor injectivity, this map is an isomorphism from~$G_1$ to~$G_2$. Similarly,~$G_1$ and~$G_3$ are isomorphic.

Finally, note that the map that sends~$(g_1,g_2,g_3)$ to~$(g_1,\varphi^{-1}(g_2)^{-1},g_1^{-1}\varphi^{-1}(g_2) )$ is an isomorphism from~$\Delta$ to~$\{(a,b,c) \in (G_1)^{3}\;|\;abc = 1\}$.
\end{proof}

We now drop the requirement for the group to be~$2$-factor surjective.
Our next examples of 2-factor injective subdirect products will be non-Abelian.

\begin{example}\label{ex:diagonal}
 Let $G_1 = H \rtimes K$ be a semidirect product with an Abelian normal subgroup $H$.
 Then
 \begin{equation}
   \label{eq-cfi-group}
   \Delta = \{(ak, bk, ck) \in (G_1)^{3}\;|\;a,b,c \in H,\,k \in K,\,abc = 1\}. 
  \end{equation}

is a  2-factor injective subdirect product of~$(G_1)^3$. To see this we verify
 that~$\Delta$ is closed under multiplication. 
Let $d = (ak, bk, ck), d' = (a'k', b'k', c'k') \in \Delta$.
Then
 \begin{align*}
  dd' &= (ak, bk, ck)(a'k', b'k', c'k')\\
      &= (aka'k', bkb'k', ckc'k')\\
      &= (a(ka'k^{-1})kk', b(kb'k^{-1})kk', c(kc'k^{-1})kk')
 \end{align*}
 and $a(ka'k^{-1})b(kb'k^{-1})c(kc'k^{-1}) = (abc)k(a'b'c')k^{-1} = 1$ implying $dd' \in \Delta$.
 So $\Delta \leq (G_1)^{3}$.
 The fact that~$\Delta$ is a subdirect product and  2-factor injective follows directly from the definition. 
\end{example}

\begin{example}
\label{ex:degenerate}
As a next example suppose~$G_1 = H_2 \times H_3$,~$G_2 = H_1\times H_3$ and~$G_3 = H_1\times H_2$ with arbitrary finite groups~$H_i$. 
Then the group consisting of the set\[\{((h_2,h_3),(h_1,h_3),(h_1,h_2))\mid h_i\in H_i\}\] is a 2-factor injective subdirect product of~$G_1\times G_2\times G_3$. 
\end{example}

Finally, it is not difficult to construct subdirect products that are not~$2$-factor injective by considering extensions of the factors.

\begin{example}
Let~$\Delta \leq G_1\times G_2\times G_3$ be a subdirect product and let~$\widetilde{G_1}$ be a surjective homomorphism~$\kappa\colon \widetilde{G_1}\rightarrow G_1$.
Then~$\{(g_1,g_2,g_3)\in \widetilde{G_1}\times G_2 \times G_3\mid (\kappa(g_1),g_2,g_3)\in \Delta\}$ is a subdirect product of~$\widetilde{G_1}\times G_2\times G_3$ that is not~$2$-factor injective if~$\kappa$ is not injective.
\end{example} 

\subsection{The structure of subgroups of 3-factor direct products}
\label{subsec:three-factors}

We now analyze the general case, showing that it must essentially be a combination of the examples presented above. We first argue that we can focus our attention on 2-factor injective subdirect products. 

\begin{lemma}
 \label{lem:from-3-factor-to-2-factor-injective}
 Let~$\Delta\leq  G_1\times G_2\times G_3$ be a subdirect product.
 Let~$N_i = \pi_i(\ker(\psi_i))$ for~$i\in \{1,2,3\}$. Then~$\Delta' = \Delta/ (N_1\times N_2\times N_3)$ is a 2-factor injective subdirect product and~$\Delta = \{(g_1,g_2,g_3) \mid (g_1 N_1,g_2N_2,g_3N_3)\in \Delta'\}$.
\end{lemma}

Thus in the following suppose~$\Delta$ is a 2-factor injective subdirect product of~$G_1\times G_2 \times G_3$.

Let~$H_i = \ker(\pi_i)\cap \Delta = \{(g_1,g_2,g_3) \in \Delta\mid g_i = 1 \}$. Then~$H = \langle H_1,H_2,H_3\rangle$ is a normal subgroup of~$\Delta$.

\begin{lemma}\label{lem:hi:and:hj:commute}

For~$i,j\in \{1,2,3\}$ with~$i\neq j$ we have~$[H_i,H_j] = 1$ that is, all elements in~$H_i$ commute with all elements in~$H_j$.
\end{lemma}

\begin{proof}
 Without loss of generality assume $i = 1$ and~$j=2$.
 For~$(1,g_2,g_3)\in H_1$ and~$(h_1,1,h_3)$ in~$H_2$ we get
 that~$(1,g_2,g_3)^{(h_1,1,h_3)} = (1,g_2,g_3^{h_3})$. By 2-factor injectivity we conclude that~$g_3^{h_3} = g_3$ and thus the two elements commute.
\end{proof}

Define~$M_i \coloneqq \pi_{i}(H_k) \cap \pi_i(H_j)$, where~$j$ and~$k$ are chosen so that~$\{i,j,k\} = \{1,2,3\}$.

\begin{lemma}\label{lem:canonical:iso:between:groups}
Let~$i,j,k$ be integers such that~$\{i,j,k\} = \{1,2,3\}$. Then there is a canonical isomorphism~$\varphi \coloneqq \varphi^i_{j,k}$ from~$\pi_j(H_i)$ to~$\pi_k(H_i)$ that maps~$M_j$ to~$M_k$.
\end{lemma}
\begin{proof}
Assume without loss of generality that~$i =1, j=2$ and~$k = 3$.
Define a map~$\varphi \colon \pi_2(H_1)\to \pi_3(H_1)$ such that~$(1,g_2,\varphi(g_2)^{-1})\in \Delta$ for all~$g_2\in \pi_2(H_1)$. Such a map exists and is well defined since~$\Delta$ is a 2-factor injective subdirect product.
Suppose~$g_2\in M_2$ then~$(1,g_2,\varphi(g_2)^{-1}) \in \Delta$ and there is a~$g_1$ such that~$(g_1,g_2,1)\in \Delta$.
Then~$(1,g_2,\varphi(g_2)^{-1}) (g_1,g_2,1)^{-1} = (g_1^{-1},1,\varphi(g_2)^{-1})$ so~$\varphi(g_2) \in M_3$. It follows by symmetry that all~$M_i$ are isomorphic and that~$\varphi|_{M_2}$ is an isomorphism from~$M_2$ to~$M_3$.
\end{proof}

Note that the canonical isomorphisms behave well with respect to composition. In particular we have~$\varphi^{i}_{j,k}= (\varphi^i_{k,j})^{-1}$ and~$ \varphi^{i}_{j,k}|_{M_j} \circ \varphi^{j}_{k,i}|_{M_k} = \varphi^{k}_{j,i}|_{M_j}$. This implies for example that the composition of the canonical isomorphism from~$M_1$ to~$M_2$ and the canonical isomorphism from~$M_2$ to~$M_3$ is exactly the canonical isomorphism from~$M_1$ to~$M_3$.
We can thus canonically identify the subgroups~$M_1$,~$M_2$ and~$M_3$ with a fixed subgroup~$M$. 

Moreover,  we can canonically associate the elements in~$H_1$ with the elements in~$\pi_2(H_1)$ and with the elements in~$\pi_3(H_1)$ by associating~$(1,g_2,\varphi(g_2)^{-1})\in H_1$ with the element~$g_2\in G_2$ and~$\varphi(g_2)$ in~$G_3$. Similarly we can associate elements~$(\varphi(g_3)^{-1},1,g_3)\in H_2$ with~$g_3\in G_3$ and~$\varphi(g_3)\in G_1$ and also associate~$(g_1,\varphi(g_1)^{-1},1)\in H_3$ with~$g_1\in G_1$ and~$\varphi(g_1)\in G_2$.

\begin{maintheoremrestated}
Let $\Delta \leq G_1\times G_2\times G_3$ be a 2-factor injective subdirect product. Then there is a normal subgroup of~$H\trianglelefteq \Delta$ with~$[\pi_i(\Delta) \colon \pi_i(H)] = [\Delta \colon H]$ for~$i\in\{1,2,3\}$ and~$H$ is isomorphic to a group of the following form: there are three finite groups~$H_1,H_2,H_3$ that all have an Abelian subgroup~$M$ contained in their the center such that~$H$ is isomorphic to the factor group of 
\[\{((h_2,{h'_3}^{-1}),(h_3,{h'_1}^{-1}),(h_1,{h'_2}^{-1}))\mid  h_i,h'_i\in H_i, h_i {h'_i}^{-1} \in M, h_1 {h'_1}^{-1}h_2 {h'_2}^{-1}h_3 {h'_3}^{-1}=1\},\] 
by the normal subgroup~$\{((m_1,m_1),(m_2,m_2),(m_3,m_3))\mid m_i\in M\}$. 
\end{maintheoremrestated}

\begin{proof}
As before define~$H_i = \ker(\pi_i) = \{(g_1,g_2,g_3) \in \Delta\mid g_i = 1 \}$ and~$H = \langle H_1,H_2,H_3\rangle$.
By Lemma~\ref{lem:canonical:iso:between:groups} and the comment about the compatibility of the isomorphism between the groups we can canonically associate the elements of~$M_i$ with those of~$M_j$. Moreover we can assume that there is an Abelian group~$M$ that is isomorphic to the intersection of every pair of~$\{H_1,H_2,H_3\}$. All elements of~$H$ commute with all elements in such an intersection.

If~$(g_1,g_2,g_3)$ is an element of~$H$ 
then~$g_i$ can be written as~$c_i \cdot b_i^{-1}$ with~$c_i \in C_i \coloneqq \pi_i(H_{i+1})$ and~$b_i \in B_i \coloneqq \pi_{i}(H_{i+2})$, where indices will always be taken modulo 3.
We can set~$(g_1,g_2,g_3) = (c_1b_1^{-1},c_2b_2^{-1},c_3b_3^{-1})$.
Since~$c_1 \in \pi_1(H_2)$ we  conclude that~$(c_1^{-1},1, \varphi^2_{1,3}(c_1))\in H$.
We also have~$(b_1,\varphi^3_{1,2}({b_1})^{-1},1)\in H$. This implies that

 \[(g_1,g_2,g_3) (c_1^{-1},1,\varphi^2_{1,3}(c_1)) (b_1,\varphi^3_{1,2}({b_1}^{-1}),1) 
 = (1,c_2 {b_2}^{-1} \varphi^3_{1,2}({b_1}^{-1}), c_3 {b_3}^{-1}  \varphi^2_{1,3}(c_1)) \in H.\]
 
We thus see by looking at the third component that~$c_3 {b_3}^{-1}  \varphi^2_{1,3}(c_1) \in \pi_3(H_1)$ which implies that~${b_3}^{-1}  \varphi^2_{1,3}(c_1)\in \pi_3(H_1)$ and thus~${b_3}^{-1} \varphi^2_{1,3}(c_1) \in M_3$. By symmetry we conclude that
\begin{equation}
 \label{eq:two-factor-eq}
 {b_{i+1}}^{-1} \varphi^i_{{i-1},{i+1}}(c_{i-1}) \in M_{i+1} \text{ for } i\in \{1,2,3\}.
\end{equation}
Looking at the second component, we also see that~$ c_2 {b_2}^{-1}\varphi^3_{1,2}({b_1}^{-1}) \in \pi_2(H_1)$ and conclude that~$\varphi^1_{2,3}(c_2 {b_2}^{-1} \varphi^3_{1,2}({b_1}^{-1}))^{-1} = c_3 {b_3}^{-1}  \varphi^2_{1,3}(c_1)$.
Recalling that~$H_i$ and~$H_j$ commute for~$i\neq j$, we thus conclude that
\begin{equation}
 \label{eq:abelian-eq-system}
 c_3{\varphi^1_{2,3}(b_2)}^{-1} \cdot  \varphi^2_{1,3}(c_1){b_3}^{-1}\cdot \varphi^1_{2,3}(c_2 \varphi^3_{1,2} ({b_1}^{-1})) = 1.
\end{equation}
Thus, since all involved isomorphisms are compatible we can reinterpret this equation in~$M$ and we see that~$c_3{(b_2)}^{-1} \cdot  c_1({b_3})^{-1}\cdot c_2 ({b_1})^{-1}  = 1$.

Suppose that~$(g_1,g_2,g_3) = (\widehat{c_1} \widehat{b_1}^{-1}, \widehat{c_2} \widehat{b_2}^{-1},\widehat{c_3} \widehat{b_3}^{-1})$ is a second representation of the element~$(g_1,g_2,g_3)$ of~$H$. Then we have for the first component that~$\widehat{c_1} \widehat{b_1}^{-1} = c_1 {b_1}^{-1}$ and thus~$c_1^{-1}\widehat{c_1} = {b_1}^{-1} \widehat{b_1}\in M_1$ or, in other words,~$\widehat{c_1} m_1 = c_1$ and~$\widehat{b_1}^{-1}m_1 = b_1^{-1}$ for some element~$m_1\in M$. Similarly there are elements~$m_2$ and~$m_3$ for the other components.
We conclude that the map sending~$(g_1,g_2,g_3)$ to~$((c_1, {b_1}^{-1}), (c_2, {b_2}^{-1}),(c_3, {b_3}^{-1}))$ is a homomorphism from~$H$ to the group described in the theorem.

It remains to show injectivity of this homomorphism. Suppose~$(g_1,g_2,g_3)$ is mapped to the trivial element. This implies for the first component of the image~$((c_1, {b_1}^{-1}), (c_2, {b_2}^{-1}),(c_3, {b_3}^{-1}))$ that~$(c_1,{b_1}^{-1}) = (m_1,m_1)$ for some~$m_1$ implying that~$g_1 = m_1 {m_1}^{-1} = 1$. Repeating the same argument for the other components we see that the homomorphism is an isomorphism.
\end{proof}

\subsection{Correspondence theorems}

We now investigate the possibility of having a correspondence theorem in the style of~Theorem~\ref{thm:goursat:correspondence} for 3 factors.
As before, we can readily reduce to the case of~2-factor injective subdirect products.

\begin{lemma}
 \label{lem:3-factor-correspondence-to-2-factor-injective}
 There is a natural one-to-one correspondence between subdirect products of~$G_1\times G_2\times G_3$ and the tuples~$(N_1,N_2,N_3,\Delta')$, where $N_i \unlhd G_i$ for every $i \in \{1,2,3\}$ and $\Delta'$ is a 2-factor injective subdirect product of~$G_1/N_1\times G_2/N_2\times G_3/N_3$.
\end{lemma}

\begin{proof}
 Let~$\Delta\leq  G_1\times G_2\times G_3$ be a subdirect product.
 Choose~$N_i = \pi_i(\ker(\psi_i))$ for~$i\in \{1,2,3\}$ and let~$\Delta' = \Delta/ (N_1\times N_2\times N_3)$.
 Then~$N_i \unlhd G_i$, because $\Delta$ is a subdirect product, and~$\Delta'$ is 2-factor injective by Lemma \ref{lem:from-3-factor-to-2-factor-injective}.
 
 Conversely, let~$N_i \unlhd G_i$ and let~$\Delta'$ be a 2-factor injective subdirect product of~$G_1/N_1\times G_2/N_2\times G_3/N_3$.
 Then~$\Delta = \{(g_1,g_2,g_3) \mid (g_1N_1,g_2N_2,g_3N_3) \in \Delta'\}$ is subdirect product of~$G_1 \times G_2 \times G_3$.
\end{proof}

Suppose~$\Delta$ is a 2-factor injective subdirect product. Then we can define for~$i\in\{1,2,3\}$ the groups~$H_i= \{(g_1,g_2,g_3) \in \Delta \mid g_i = 1\}$ and with them the groups $B_i = \pi_i(H_{i+2})$ and $C_i = \pi_i(H_{i+1})$. As we will see, the canonical isomorphism~$\varphi^i \coloneqq \varphi^{i+2}_{i,i+1}$ that exists by Lemma~\ref{lem:canonical:iso:between:groups} can be extended to an isomorphism from~$G_i/C_i$ to~$G_{i+1}/B_{i+1}$. We would like to have a correspondence theorem in the style of~Theorem~\ref{thm:goursat:correspondence} for 3 factors.
For this, in principle, we would like to relate the 2-factor injective subdirect products of~$G_1,G_2,G_3$ to the tuples~\[(B_1,B_2,B_3,C_1,C_2,C_3,\varphi_1,\varphi_2,\varphi_3)\] that satisfy certain consistency properties. However, in general, neither is it clear which consistency properties to choose so that every tuple corresponds to a subdirect product, nor do distinct subdirect products always correspond to distinct tuples. Indeed, in~\cite{GenGoursat} the authors describe two distinct Abelian subdirect products of the same group~$G_1\times G_2 \times G_3$ for which the corresponding tuples agree.

In the light of that we content ourselves with study two special cases, namely those where~$\Delta  = H$ and those where~$M_i = B_i\cap C_i = 1$.

\begin{theorem}
 \label{thm:correspondence-degenerate-non-diagonal}
 There is a natural one-to-one correspondence between subdirect products of~$\Delta  \leq G_1\times G_2\times G_3$ which are 2-factor injective satisfying~$\Delta = H$, 
 and tuples of the form~$(B_1,B_2,B_3,C_1,C_2,C_3,\varphi_1,\varphi_2,\varphi_3)$ for which for all~$i\in \{1,2,3\}$ (indices taken modulo 3) we have 
 \begin{enumerate}
  \item $B_i,C_i \trianglelefteq G_i$,
  \item $B_iC_i = G_i$,\label{item:non:diagonal}
  \item $B_i \stackrel{\varphi_i}{\cong} C_{i+1}$,
  \item $[B_i, C_i]= 1$,
  \item $\varphi_i(B_i\cap C_i) = B_{i+1}\cap C_{i+1}$,
  \item $\varphi_3|_{B_3\cap C_3}\circ\varphi_2|_{B_2\cap C_2}\circ \varphi_1|_{B_1\cap C_1} = \text{id}$.\label{item:concatenation}
 \end{enumerate}
\end{theorem}

\begin{proof}
 Let $\Delta \leq G_1 \times G_2 \times G_3$ be a subdirect product. Define $H_i \coloneqq \{(g_1,g_2,g_3) \in \Delta \mid g_i = 1\}$ for $i \in \{1,2,3\}$ and $H = \langle H_1,H_2,H_3 \rangle$. Suppose that $H = \Delta$.
 For $i \in \{1,2,3\}$, define $B_i = \pi_i(H_{i+2})$ and $C_i = \pi_i(H_{i+1})$.
 Clearly $B_i,C_i \unlhd G_i$.
 By Lemma \ref{lem:hi:and:hj:commute} we get that $[B_i,C_i] = 1$ and $B_iC_i = \pi_i(H)$.
 The assumption $\Delta =H$ implies $B_iC_i = G_i$.
 By Lemma \ref{lem:canonical:iso:between:groups} the groups $B_i$ and $C_{i+1}$ are isomorphic via an isomorphism $\varphi_i = \varphi_{i,i+1}^{i+2}$, which maps $M_i = B_i \cap C_i$ to $M_{i+1} = B_{i+1} \cap C_{i+1}$.
 Finally Property~\ref{item:concatenation} follows directly from the comment below Lemma \ref{lem:canonical:iso:between:groups}.
 This gives us the tuple~$(B_1,B_2,B_3,C_1,C_2,C_3,\varphi_1,\varphi_2,\varphi_3)$ with the desired properties.
 
 On the other hand suppose we are given a tuple~$(B_1,B_2,B_3,C_1,C_2,C_3,\varphi_1,\varphi_2,\varphi_3)$ with the desired properties.
 Let $M_i = B_i \cap C_i$.
 Define
 \[
  \Delta = \left\{(g_1,g_2,g_3) \in G_1 \times G_2 \times G_3 \;\bigg|\; \begin{aligned}
                                                                          &g_i = c_ib_i^{-1} \text{ for } b_i \in B_i,c_i \in C_i,\; c_{i+1}M_{i+1} = \varphi_i(b_iM_i)\\
                                                                          &\text{and } c_3{\varphi_{2}(b_2)}^{-1} \cdot  \varphi_{3}^{-1}(c_1){b_3}^{-1}\cdot \varphi_{2}(c_2 \varphi_{1} ({b_1}^{-1})) = 1
                                                                         \end{aligned}\right\}.
 \]
 For $i \in \{1,2,3\}$ suppose $g_i = c_ib_i^{-1} = c_i'b_i'^{-1}$.
 Then there is some $m_i \in M_i$ with $b_i = b_i'm_i$ and $c_i = c_i'm_i$.
 Hence, $b_iM_i = b_i'M_i$ and $c_iM_i = c_i'M_i$.
 Furthermore, we have
 \begin{align*}
      &c_3'{\varphi_{2}(b_2')}^{-1} \cdot  \varphi_{3}^{-1}(c_1'){b_3'}^{-1}\cdot \varphi_{2}(c_2' \varphi_{1} ({b_1'}^{-1}))\\
  =\; &c_3m_3^{-1}{\varphi_{2}(b_2m_2^{-1})}^{-1} \cdot  \varphi_{3}^{-1}(c_1m_1^{-1}){(b_3m_3^{-1})}^{-1}\cdot \varphi_{2}(c_2m_2^{-1} \varphi_{1} ({(b_1m_1^{-1})}^{-1}))\\
  =\; &c_3{\varphi_{2}(b_2)}^{-1} \cdot  \varphi_{3}^{-1}(c_1){b_3}^{-1}\cdot \varphi_{2}(c_2 \varphi_{1} ({b_1}^{-1})),
 \end{align*}
 so the membership in $\Delta$ is independent from the representation of $g_i \in G_i$.
 Also, it is easy to check that $\Delta$ is closed under multiplication because $[B_i, C_i]= 1$.
 The group $\Delta$ is a subdirect product, since for $g_1 = c_1b_1^{-1}$ we have $(c_1b_1^{-1},\varphi_1(b_1^{-1})^{-1}b_2^{-1},\varphi_{3}^{-1}(c_1)\varphi_2(b_2)) \in \Delta$, and it can be checked that the group is 2-factor injective.
 Define $H_i = \{(g_1,g_2,g_3) \in \Delta \mid g_i = 1\}$ for $i \in \{1,2,3\}$ and $H = \langle H_1,H_2,H_3 \rangle$.
 Then $B_i = \pi_i(H_{i+2})$ and $C_i = \pi_i(H_{i+1})$, which means that $H = \Delta$ by Property~\ref{item:non:diagonal} and Theorem~\ref{thm:characterization-two-factor-injective}.
 Finally, it can be checked that $\varphi_i = \varphi_{i,i+1}^{i+2}$.
 
 It remains to show that for each 2-factor injective subdirect product $\Delta \leq G_1 \times G_2 \times G_3$ with~$\Delta = H$, the group~$\Delta$ is of the form described above.
 But this follows from Theorem \ref{thm:characterization-two-factor-injective}, Equation (\ref{eq:two-factor-eq}) and (\ref{eq:abelian-eq-system}).
\end{proof}

The previous theorem shows that for the subdirect products with~$\Delta = H$ we can devise a correspondence theorem. 
As a second case, on the other end of the spectrum, we can also devise a correspondence theorem if the Abelian part that interlinks the three components is trivial. In fact this case corresponds to the case discussed by Bauer, Sen, and Zvengrowski~\cite[5.1 Remark]{GenGoursat}, who already suspect that a theorem like the previous can be obtained. We remark that
Example \ref{ex:degenerate} from the previous section is of this form.
In fact, we can already conclude from Theorem \ref{thm:correspondence-degenerate-non-diagonal} that for every group $\Delta$, where the Abelian part interlinking the components is trivial, the group $H$ essentially has the form of Example \ref{ex:degenerate}.

\begin{definition}
 Let $\Delta$ be a subdirect product of $G_1 \times G_2 \times G_3$.
 We say $\Delta$ is \emph{degenerate} if $\pi_i(\ker(\pi_{i+1})) \cap \pi_i(\ker(\pi_{i+2})) = \pi_i(\ker(\psi_i))$ (i.e.\ $M_i = 1$) for some, and thus every, $i \in \{1,2,3\}$.
\end{definition}

\begin{lemma}
 \label{la:trivial-mi-upper-bound-construction}
 For $i \in \{1,2,3\}$ let $B_i,C_i \unlhd G_i$, such that $B_i \cap C_i = 1$ and $[B_i,C_i] = 1$.
 Furthermore let $G_i/C_i \stackrel{\varphi_i}{\cong} G_{i+1}/B_{i+1}$ and suppose $\varphi_i(B_iC_i) = C_{i+1}B_{i+1}$ and $\varphi_3(\varphi_2(\varphi_1(g_1B_1C_1))) = g_1B_1C_1$ for all $g_1 \in G_1$.
 Define
 \begin{equation}
  \label{eq:trivial-m-i}
  \Delta = \{(g_1,g_2,g_3) \in G_1 \times G_2 \times G_3 \mid \varphi_i(g_iC_i) = g_{i+1}B_{i+1}\}.
 \end{equation}
 Then $\pi_i(\Delta) = G_i$ and $\Delta$ is a degenerate 2-factor injective subdirect product.
\end{lemma}

\begin{proof}
 We first show, that $\Delta$ is closed under multiplication.
 Let $(g_1,g_2,g_3),(g_1',g_2',g_3') \in \Delta$.
 Then $\varphi_i(g_ig_i'C_i) = \varphi_i(g_iC_i)\varphi_i(g_i'C_i) = g_{i+1}g_{i+1}'B_{i+1}$ for all $i \in \{1,2,3\}$, so $(g_1g_1',g_2g_2',g_3g_3') \in \Delta$.
 Let $E_1 = \langle B_1,C_1 \rangle$ and pick $e_1 \in E_1$.
 The element $e_1$ can uniquely be written as $e_1 = b_1c_1$ with $b_1 \in B_1, c_1 \in C_1$.
 For each $i \in \{1,2,3\}$ define $\varphi_i^{*}\colon B_i \rightarrow C_{i+1}$ with $\varphi_i^{*}(b_i) = c_{i+1}$ for the unique $c_{i+1} \in C_{i+1}$ with $\varphi_i(b_iC_i) = c_{i+1}B_{i+1}$.
 Then $(b_1c_1, b_2\varphi_1^{*}(b_1),{(\varphi_3^{*})}^{-1}(c_1)\varphi_2^{*}(b_2)) \in \Delta$.
 So $E_1 \leq \pi_1(\Delta)$.
 The argument for the other components is analogous.
 Now let $n_1^{1},\dots,n_t^{1}$ be a transversal of $E_1$ in $G_1$.
 Let $n_i^{2}B_2 = \varphi_1(n_i^{1}C_1)$ and $n_i^{3}C_3 = \varphi_3^{-1}(n_i^{1}B_1)$ for $i \in \{1,\dots,t\}$.
 Then $\varphi_2(n_i^{2}C_2) \subseteq n_i^{3}E_3$ and hence there is some $b_i^{2} \in B_2$ with $\varphi_2(n_i^{2}b_i^{2}C_2) = n_1^{3}B_3$.
 So $(n_i^{1},n_i^{2}b_i^{2},n_i^{3}) \in \Delta$ and $G_1 \leq \pi_1(\Delta)$.
 
 It remains to prove that $\Delta$ is 2-factor injective.
 Let $(g_1,g_2,g_3) \in \Delta$ with $g_2 = g_3 = 1$.
 Then $g_1 \in B_1$, because $\varphi_3(C_3) = B_1 = g_1B_1$, and $g_1 \in C_1$, because $\varphi_1(g_1C_1) = g_2B_2 = B_2$.
 So $g_1 = 1$.
 Again, the argument for the other components is analogous.
\end{proof}

\begin{lemma}
\label{la:trivial-mi-upper-bound}
 Let $\Delta$ be a 2-factor injective subdirect product of $G_1 \times G_2 \times G_3$.
 Furthermore, let $H_i = \{(g_1,g_2,g_3) \in \Delta \mid g_i = 1\}$ for $i \in \{1,2,3\}$ and $H = \langle H_1,H_2,H_3 \rangle$.
 Define $B_i = \pi_i(H_{i+2})$ and $C_i = \pi_i(H_{i+1})$.
 Suppose $B_i \cap C_i = 1$ for all $i \in \{1,2,3\}$.\\
 Then there are canonical isomorphisms~$\varphi_1,\varphi_2,\varphi_3$ with $G_i/C_i \stackrel{\varphi_i}{\cong} G_{i+1}/B_{i+1}$ and $\varphi_i(B_iC_i) = C_{i+1}B_{i+1}$ such that $\varphi_3(\varphi_2(\varphi_1(g_1B_1C_1))) = g_1B_1C_1$ for all $g_1 \in G_1$.
 Furthermore $\Delta$ is given by Equation (\ref{eq:trivial-m-i}).
\end{lemma}

\begin{proof}
 For $i \in \{1,2,3\}$ define a homomorphism $\varphi_i\colon G_i/C_i \rightarrow G_{i+1}/B_{i+1}$ by setting $\varphi_i(g_iC_i) = g_{i+1}B_{i+1}$ if  $(g_1,g_2,g_3) \in \Delta$ for some~$g_i\in G_i$.
 We first have to show that $\varphi_i$ is well-defined.
 Without loss of generality consider $i=1$ and let $(g_1,g_2,g_3), (g_1',g_2',g_3') \in \Delta$ with $g_1C_1 = g_1'C_1$.
 Then there is a $(c,1,h_2) \in \Delta$ with $g_1'c = g_1$.
 We obtain $(g_1',g_2',g_3')(c,1,h_2)(g_1,g_2,g_3)^{-1} = (1,g_2'g_2^{-1},g_3'')$ for some $g_3'' \in G_3$ and hence, $g_2B_2 = g_2'B_2$.
 So $\varphi_i$ is well-defined. Since~$\Delta$ is a subdirect product, ~$\varphi_i$ is a surjective homomorphism.
 Suppose $\varphi_1(g_1C_1) = B_2$. Then $(g_1c_1, b_2, g_3) \in \Delta$ for some $c_1 \in C_1, b_2 \in B_2$ and $g_3 \in G_3$.
 Also there is $h_3 \in G_3$ with $(1,b_2,h_3) \in \Delta$ and hence, $(g_1c_1,1,g_3h_3^{-1}) \in \Delta$ implying that $g_1 \in C_1$.
 So $G_i/C_i \stackrel{\varphi_i}{\cong} G_{i+1}/B_{i+1}$.
 
 For every $b_1 \in B_1$ there is a $c_2 \in C_2$ with $(b_1,c_2,1) \in \Delta$ and $\varphi_1(b_1C_1) = c_2B_2 \in C_2B_2$.
 By symmetry it follows that $\varphi_i(B_iC_i) = C_{i+1}B_{i+1}$ for all $i \in \{1,2,3\}$.
 Now let $\Delta'$ be the group defined in Equation (\ref{eq:trivial-m-i}).
 Clearly $\Delta \leq \Delta'$ by the definition of $\varphi_i$ for $i \in \{1,2,3\}$.
 So let $(g_1',g_2',g_3') \in \Delta'$.
 Since $\Delta$ is subdirect there is a $(g_1',g_2,g_3) \in \Delta$ with $g_2B_2 = g_2'B_2$.
 So we can assume that $g_2 = g_2'$.
 But then, by 2-factor injectivity of $\Delta'$, we get that $g_3 = g_3'$.
 
 Finally for $(g_1,g_2,g_3) \in \Delta$ we have that $\varphi_i(g_iB_iC_i) = \varphi_i(g_iC_i)\varphi_i(B_iC_i) = g_{i+1}C_{i+1}B_{i+1} = g_{i+1}B_{i+1}C_{i+1}$.
 So $\varphi_2(\varphi_1(g_1B_1C_1)) = \varphi_3^{-1}(g_1B_1C_1)$ for all $g_1 \in G_1$.
\end{proof}

\begin{theorem}
 \label{thm:correspondence-degenerate-two-factor-injective}
 There is a natural one-to-one correspondence between degenerate 2-factor injective subdirect products of $G_1 \times G_2 \times G_3$
 and the tuples $(B_1,B_2,B_3,C_1,C_2,C_3,\varphi_1,\varphi_2,\varphi_3)$ for which for all~$i\in \{1,2,3\}$ (indices taken modulo 3) we have 
 \begin{enumerate}
  \item $B_i,C_i \trianglelefteq G_i$,
  \item $B_i \cap C_i = 1$,
  \item $[B_i, C_i]= 1$,
  \item $G_i/C_i \stackrel{\varphi_i}{\cong} G_{i+1}/B_{i+1}$,
  \item $\varphi_i(B_iC_i) = C_{i+1}B_{i+1}$,
  \item $\varphi_3(\varphi_2(\varphi_1(g_1B_1C_1))) = g_1B_1C_1$ for all $g_1 \in G_1$.
 \end{enumerate}
\end{theorem}

\begin{proof}
 The statement follows from Theorem \ref{thm:characterization-two-factor-injective}, Lemma \ref{la:trivial-mi-upper-bound-construction} and \ref{la:trivial-mi-upper-bound}.
\end{proof}

By combining Theorem \ref{thm:correspondence-degenerate-two-factor-injective} with Lemma \ref{lem:3-factor-correspondence-to-2-factor-injective} we obtain the following correspondence result for degenerate subdirect products.

\begin{corollary}
 \label{cor:correspondence-degenerate}
 There is a natural one-to-one correspondence between degenerate subdirect products of $G_1 \times G_2 \times G_3$
 and the tuples $(N_1,N_2,N_3,B_1,B_2,B_3,C_1,C_2,C_3,\varphi_1,\varphi_2,\varphi_3)$ for which for all~$i\in \{1,2,3\}$ (indices taken modulo 3) we have 
 \begin{enumerate}
  \item $N_i \trianglelefteq G_i$,
  \item $B_i,C_i \trianglelefteq G_i/N_i$,
  \item $B_i \cap C_i = 1$,
  \item $[B_i, C_i]= 1$,
  \item $(G_i/N_i)/C_i \stackrel{\varphi_i}{\cong} (G_{i+1}/N_{i+1})/B_{i+1}$,
  \item $\varphi_i(B_iC_i) = C_{i+1}B_{i+1}$,
  \item $\varphi_2(\varphi_1(g_1B_1C_1)) = \varphi_3^{-1}(g_1B_1C_1)$ for all $g_1 \in G_1/N_1$.
 \end{enumerate}
\end{corollary}

We conclude with the particular case in which~$\pi_i(\Delta)$ has a complement in~$G_i$ for some~$i\in \{1,2,3\}$. Example~\ref{ex:diagonal} described in the previous section is of this form.
For this case we only obtain an injection to tuples, rather than a one-to-one correspondence. We will exploit having this injection in the next section for small special cases in our analysis of subdirect products of symmetric groups.

\begin{theorem}
 \label{thm:injection:semi:direct}
 Suppose~$G_1 = E_1 \rtimes K$ is a semidirect product. There is an injective mapping  from the set of 2-factor injective subdirect products~$\Delta$ of $G_1 \times G_2 \times G_3$ with~$\pi_1(H) = E_1$ and~$B_1 = C_1$
 to the tuples $(\kappa,\iota)$ where~$G_2 \stackrel{\kappa}{\cong} G_{3}$ and~$\iota$ is an automorphism of~$G_1$ that fixes~$K$ as a set.
 Moreover if~$(\kappa,\iota)$ is in the image of this mapping then~$(\kappa,\iota')$ is also in the image for every automorphism~$\iota'$ of~$G_1$ that fixes~$K$ as a set. 
\end{theorem}
\begin{proof}

Let~$\Delta$ be a 2-factor injective subdirect product of $G_1 \times G_2 \times G_3$ satisfying~$\pi_1(H) = E_1$ and~$B_1 = C_1$. 

For every element~$g_2 \in G_2$ 
 there is exactly one element~$(k_1,g_2,g_3)\in \Delta$ with~$k_1\in K$.
We obtain a well defined isomorphism~$\kappa$ from~$G_2$ to~$G_{3}$.

Suppose now that~$\Delta'$ is a second 2-factor injective subdirect product $G_1 \times G_2 \times G_3$ satisfying~$\pi_1(H) = E_1$ and~$B_1 = C_1$ for which we obtain the same isomorphism~$\kappa$.
Then we can construct an automorphism~$\iota$ of~$G_1$ as follows. For every~$g_2\in G_2$ there is exactly one~$(k,g_2,\kappa(g_2))\in \Delta$. There is also an element of the form~$(k',g_2,\kappa(g_2)) \in \Delta'$. We define~$\iota_K\colon K\rightarrow K$ so that it maps~$k$ to~$k'$, this gives us an automorphism~$\iota_K$ of~$K$. For every~$e\in E_1$ there is an element~$(e,g_2,1)\in \Delta$. There is also an element~$(e',g_2,1)\in \Delta'$ and define the map~$\iota_E\colon E_1\rightarrow E_1$ by mapping~$e$ to~$e'$. Then the map~$\iota_E$ is an automorphism of~$E$.

We claim that the map that sends~$e\cdot k$ to~$\iota_E(e)\cdot \iota_K(k)$ is an automorphism of~$G_1$.
 To see this suppose~$a = (e_1,h,1)(k,g,\kappa(g))$ and~$\overline{a} = (\overline{e_1},\overline{h},1)(\overline{k},\overline{g},\kappa(\overline{g}))$ are two elements in~$\Delta$.
Then~$\iota(a) = (\iota_E(e_1),h,1)(\iota_K(k),g,\kappa(g))$ and~$\iota(\overline{a}) = (\iota_E(\overline{e_1}),\overline{h},1)(\iota_K(\overline{k}),\overline{g},\kappa(\overline{g}))$ are elements of~$\Delta'$.

For the products we obtain that\[a \overline{a} = (e_1\overline{e_1}^{k},h \overline{h}^{g},1)(k\overline{k},g\overline{g},\kappa(g)\kappa(\overline{g}))\] and \[\iota(a) \iota(\overline{a}) = (\iota_E(e_1)\iota_E(\overline{e_1})^{\iota_K(k)},h \overline{h}^{g},1)(\iota_K(k)\overline{\iota_K(k)},g\overline{g},\kappa(g)\kappa(\overline{g})).\]
To conclude that~$\iota$ is an isomorphism we now only need to argue that~$\iota_E(\overline{e_1})^{\iota_K(k)}$ is equal to~$\iota_E(\overline{e_1}^k)$. However, this is the case since~$(\overline{e_1}^k, h^g,1) \in \Delta$ and~$(\iota_E(e_1)^{\iota_K(k)}, h^g,1)\in \Delta'$.

Now suppose that~$\Delta$ is a subdirect product with~$\pi_1(H) = B_1 = C_1$ and let~$\kappa\colon G_2\rightarrow G_3$ be defined as above.
Let~$\iota$  be an automorphism of~$G_1$ that fixes~$K$ then~$\Delta' = \{(\iota(g_1),g_2,g_3) \mid (g_1,g_2,g_3)\in \Delta \}$ is a subdirect product of~$G_1\times G_2\times G_3$. If we apply the above construction for the automorphism of~$G_1$ we reobtain~$\iota$. This shows that the construction of~$\iota$ from~$\Delta'$ and the construction of~$\Delta'$ from~$\iota$ are inverses to one another.
\end{proof}

Note that in the theorem, the isomorphism~$\kappa$ associated with a subdirect product is canonical (it only depends on the choice of~$K$) but the choice of~$\iota$ is not.

\section{Subdirect products of two or three symmetric groups}

In this section we apply the correspondence theorems to count the subdirect products of the direct product of three symmetric groups.
We first reduce this problem to counting the number of 2-factor injective subdirect products.
For finite groups $G_1,\dots,G_k$ let $\ell(G_1,\dots,G_k)$ be the number of subdirect products of $G_1 \times \dots \times G_k$.
Furthermore, for~$k=3$, we denote by $\ell_{2\text{-inj}}(G_1,\dots,G_3)$ the number of 2-factor injective subdirect products.

\begin{lemma}
 \label{la:counting-subdirect-with-2-factor-injective}
 Let $G_1,G_2,G_3$ be finite non-trivial groups. Then
 \begin{align*}
  \ell(G_1,G_2,G_3) &= \sum_{N_i \unlhd G_i} \ell_{2\text{-inj}}(G_1/N_1,G_2/N_2,G_3/N_3)\\
                    &= \ell(G_{1},G_{2})+\ell(G_{2},G_{3})+\ell(G_{1},G_{3})-2 + \sum_{N_i \lhd G_i} \ell_{2\text{-inj}}(G_1/N_1,G_2/N_2,G_3/N_3).
 \end{align*}
\end{lemma}

\begin{proof}
 The first equality follows from the correspondence described in Lemma \ref{lem:3-factor-correspondence-to-2-factor-injective}. The second equality follows by noting that the direct product is counted three times, so 2 has to be subtracted.
\end{proof}

We are now interested in the number $\ell(n_1,n_2,n_3) := \ell(S_{n_1},S_{n_2},S_{n_3})$, where~$S_{n_i}$ is  the symmetric group of a set with~$n_i$ elements. 
Recall, that every factor group of a symmetric group is isomorphic to a symmetric group over another set. Thus, 
by the previous lemma, it suffices to compute the numbers~$\ell(n_1,n_2)$ and ~$\ell_{2\text{-inj}}(n_1,n_2,n_3) := \ell_{2\text{-inj}}(S_{n_1},S_{n_2},S_{n_3})$.

We start by analyzing the situation for two factors.

\begin{lemma}
\label{la:counting-2-factor-sn}
Let $n_1,n_2 \geq 2$. For the number~$\ell(n_1,n_2)$ of subdirect products of~$S_{n_1}\times S_{n_2}$ we have~\[\ell(n_1,n_2) = \begin{cases}
2 & \text{ if } n_1\neq n_2 \text{ and } \{n_1, n_2\}\neq \{3,4\}, \\
8 & \text{ if } \{n_1, n_2\} = \{3,4\}, \\
n_1! + 2&  \text{ if } n_1= n_2\notin \{2,4,6\},\\
2&  \text{ if } n_1= n_2= 2,\\
n_1! + 8&  \text{ if } n_1= n_2= 4,\\
2n_1! + 2&  \text{ if } n_1= n_2= 6.\\
\end{cases}\]
\end{lemma}
\begin{proof}
We assume for our considerations that~$n_1\geq n_2$.
Let~$(S_{n_1},S_{n_2},N_1,N_2,\varphi)$ be a tuple corresponding to a subdirect product via the correspondence of Theorem~\ref{thm:goursat:correspondence}.
\begin{itemize}

\item If~$N_1 = 1$ then~$N_2 = 1$ and~$n_1= n_2$. The number of isomorphisms from~$S_{n_1}$ to~$S_{n_1}$ is~\[i(n_1) = \begin{cases}
                                                                                                                       1 &\text{ if } n_1= 2\\
                                                                                                                       2n! &\text{ if } n_1= 6\\
                                                                                                                       n! &\text{ otherwise}\\
                                                                                                                      \end{cases}.\]
This corresponds to the number of possible choices for~$\varphi$. (These are the diagonal subgroups.)

\item If~$N_1 = S_{n_1}$ then~$N_2 = S_{n_2}$. There is only one subgroup of this type. (This is the direct product).

\item If~$N_1 = A_{n_1}$ ($n_1 \geq 3$) then~$N_2 = A_{n_2}$, since the only index~$2$ subgroup that a symmetric group can have is the alternating group. There is only one subgroup of this type.

\end{itemize}

If~$n_1\neq  4$ then~$N_1 \in \{1,A_{n_1},S_{n_1}\}$, and we already considered all these cases.
Suppose now that~$n_1 =4$ and~$n_2\leq 4$. Then~$N_1\in \{1,V,A_{n_1},S_{n_1}\}$, where~$V$ is the Klein-four-group. Three of cases are considered above.
\begin{itemize}
\item If $N_1 = V$ then~$N_2 = V$ and~$n_2 = 4$ or~$N_2 = 1$ and~$n_2 = 3$. In either case there are~$6$ options for~$\varphi$.
\end{itemize}
\end{proof}

In the following we use~``$\leftmset$'' and~``$\rightmset$'' to denote multisets.

\begin{lemma}
\label{la:counting-3-factor-sn-2-factor-injective}
 Let $n_1,n_2,n_3 \geq 2$. For the number~$\ell_{2\text{-inj}}(n_1,n_2,n_3)$ of~$2$-factor injective subdirect products of~$S_{n_1}\times S_{n_2}\times S_{n_3}$ we have~\[\ell_{2\text{-inj}}(n_1,n_2,n_3) = \begin{cases}

(n_1!)^{2} &  \text{ if } n_1 = n_2 = n_3 \notin \{2,3,4,6\},\\
2 &  \text{ if } n_1 = n_2 = n_3 = 2,\\
(n_1!)^{2} + 2n_1! &  \text{ if } n_1 = n_2 = n_3 = 3,\\
(n_1!)^{2} + 6n_1! &  \text{ if } n_1 = n_2 = n_3 = 4,\\
(2n_1!)^{2} &  \text{ if } n_1 = n_2 = n_3 = 6,\\
1440 &  \text{ if } \leftmset n_1,n_2,n_3\rightmset = \leftmset 2,6,6\rightmset\\
n! &  \text{ if } \leftmset n_1,n_2,n_3\rightmset = \leftmset 2,n,n\rightmset \text{ for } n \notin \{2,6\},\\
144 &  \text{ if } \leftmset n_1,n_2,n_3\rightmset = \leftmset 3,4,4\rightmset\\
0&  \text{ otherwise.}\\

\end{cases}\]
\end{lemma}

\begin{proof}
 Suppose without loss of generality that $n_1 \geq n_2 \geq n_3$.
 Let~$\Delta \leq S_{n_1}\times S_{n_2}\times S_{n_3}$ be a $2$-factor injective subdirect product.
 Define~$H_1,H_2,H_3,H$ as in Section \ref{subsec:three-factors}.
 Let~$E_i = \pi_i(H)$.
 Then~$H \unlhd \Delta$ and~$E_i \unlhd S_{n_i}$ for~$i \in \{1,2,3\}$.
 By Theorem~\ref{thm:characterization-two-factor-injective}, it holds that~$S_{n_1}/E_1 \cong S_{n_2}/E_2 \cong S_{n_3}/E_3$.
 Also, by Theorem \ref{thm:correspondence-degenerate-non-diagonal}, for~$H$ we get a canonical tuple~$(B_1,B_2,B_3,C_1,C_2,C_3,\varphi_1,\varphi_2,\varphi_3)$ with subgroups $B_i,C_i \unlhd E_i$, such that $B_i \cap C_i$ is Abelian and $B_iC_i = E_i$.
 We obtain the following options.
 \begin{itemize}
  \item If~$E_1 = 1$ then $n_1 = n_2 = n_3$ and $E_2 = E_3 = 1$. We have $B_i = C_i = 1$ for $i \in \{1,2,3\}$
   In this case $H = 1$ and $\Delta$ is degenerate.
   By Theorem \ref{thm:correspondence-degenerate-two-factor-injective} there are $i(n_1)^{2}$ groups of this type, where $i(n_1)$ is the number of isomorphisms from $S_{n_1}$ to $S_{n_1}$.
   This corresponds to the choices for $\varphi_1$ and $\varphi_2$. For~$\varphi_3$ we get $\varphi_3^{-1} = \varphi_1 \circ \varphi_2$.
  \item If~$E_1 = V$ then $n_1 = 4$ and $S_{n_1}/E_1 \cong S_3$. In this case $\{n_1,n_2,n_3\} \subseteq \{3,4\}$.
  Let us first consider the case that~$n_3 = 3$. Then~$B_3 = C_3 = E_3 = 1$ and we can thus apply Theorem~\ref{thm:correspondence-degenerate-two-factor-injective}. 
  By Lemma \ref{lem:canonical:iso:between:groups} we conclude that $C_1 = B_2 = 1$ and $E_1 = B_1 = C_2 = V$.
  Using the correspondence given in Theorem \ref{thm:correspondence-degenerate-two-factor-injective} the number of such groups equals the number of pairs $(\varphi_1,\varphi_2)$, where $\varphi_1$ is an isomorphism from $S_4$ to $S_4$ and $\varphi_2$ is an isomorphism from $S_3$ to $S_3$.
  There are~$144$ such pairs.
  
  Next let us consider the case~$n_1 = n_2 = n_3 = 4$.
  This implies that~$B_i = C_i = E_i = V$ for all~$i\in \{1,2,3\}$. Since~$S_4$ is the semidirect product~$V\rtimes S_3$, we can apply Theorem~\ref{thm:injection:semi:direct}. 
  For every isomorphism~$\kappa$ from~$G_2$ to~$G_3$ we can find a subdirect product realizing~$\kappa$ by setting~$\Delta = \langle \{k,g_1,\kappa(g_1)\mid g_1\in G_1, k\in S_3\cap g_1V \} \cup \{(a,a^{-1},1)\mid a\in V\}\rangle $. Thus, by Theorem~\ref{thm:injection:semi:direct} the number of such subdirect products is equal to the number pairs~$(\kappa, \iota)$ where~$\kappa$ is an isomorphism from~$S_{4}$ to~$S_{4}$ and~$\iota$ is an automorphism of~$S_4$ that fixes~$V$ as a set. There are~$6n_1! = 144$ such pairs.
  \item If~$E_1 = A_{n_1}$ (and $n_1 \geq 3$) then $E_2 = A_{n_2}$ and $E_3 = A_{n_3}$.
   By applying Theorem \ref{thm:correspondence-degenerate-non-diagonal} to~$H$ it follows that either $n_1 = n_2 = n_3 = 3$ or $n_1 = n_2 > n_3 = 2$.
   In the former case $B_i = C_i = A_3 = \mathbb{Z}_3$ for $i \in \{1,2,3\}$ and in total there are $2n_1! = 12$ groups of this type by Theorem~\ref{thm:injection:semi:direct} (by the same arguments as in the previous case).
   In the latter case $B_1 = C_2 = A_{n_1}$ and $C_1 = B_2 = B_3 = C_3 = 1$.
   So $\Delta$ is degenerate and by Theorem \ref{thm:correspondence-degenerate-two-factor-injective} there are in total $i(n_1)\cdot i(2) = i(n_1)$ groups of this type, where again $i(n_1)$ is the number of isomorphisms from $S_{n_1}$ to $S_{n_1}$.
  \item If~$E_1 = S_{n_1}$ then $E_2 = S_{n_2}$ and $E_3 = S_{n_3}$.
   In this case $H = \Delta$.
   Using the correspondence described in Theorem \ref{thm:correspondence-degenerate-non-diagonal} we get that $n_1 = n_2 = n_3 = 2$ and $B_i = C_i = S_2$ for $i \in \{1,2,3\}$.
   Since there is only one isomorphism from $S_2$ to $S_2$ there is exactly one option in this case, namely the group given in Example \ref{ex:cfi-group} for $G_1 = S_2$.
 \end{itemize}
\end{proof}

\begin{corollary}
 Let $n_1 \geq n_2 \geq n_3 \geq 2$, $n_1 \geq 5$. For the number~$\ell(n_1,n_2,n_3)$ of subdirect products of~$S_{n_1}\times S_{n_2}\times S_{n_3}$ we have~\[\ell(n_1,n_2,n_3) = \begin{cases}
(n_1!)^{2} + 6n_1! + 6 &  \text{ if } n_1 = n_2 = n_3 \notin \{6\},\\
2082246 &  \text{ if } n_1 = n_2 = n_3 = 6,\\
66 &  \text{ if } n_2 = n_3 = 4,\\
18 &  \text{ if } n_2 \in\{3,4\}, n_3 = 3,\\
2886 & \text{ if } \leftmset n_1,n_2,n_3\rightmset = \leftmset 6,6,m_2\rightmset,m_2 \neq 6,\\
2m_1! + 6 & \text{ if } \leftmset n_1,n_2,n_3\rightmset = \leftmset m_1,m_1,m_2\rightmset, m_1 \neq m_2, 6 \neq m_1 \geq 5,\\
6 &  \text{ otherwise.}\\ 
\end{cases}\]
\end{corollary}

\begin{proof}
 Using Lemma \ref{la:counting-subdirect-with-2-factor-injective} and \ref{la:counting-3-factor-sn-2-factor-injective} we get
 \[\ell(n_1,n_2,n_3) = \sum_{i < j}\ell(n_{i},n_{j}) + \begin{cases}
\ell_{2\text{-inj}}(n_1,n_1,n_1) + 3\ell_{2\text{-inj}}(2,n_1,n_1)   &  \text{if } n_1 = n_2 = n_3,\\
\ell_{2\text{-inj}}(2,4,4) + \ell_{2\text{-inj}}(2,3,3) &  \text{if } n_2 = n_3 = 4,\\
\ell_{2\text{-inj}}(2,3,3) &  \text{if } n_2 \in\{3,4\}, n_3 = 3,\\
\ell_{2\text{-inj}}(2,m_1,m_1) &\begin{array}{@{}ll}
                                   \text{if} &\leftmset n_1,n_2,n_3\rightmset = \\&\leftmset m_1,m_1,m_2\rightmset\text{ for}\\&m_1 \neq m_2, m_1 \geq 5,
                                  \end{array}\\
0 &  \text{otherwise.}\\ 
\end{cases}\]
 Then apply Lemma \ref{la:counting-2-factor-sn} and \ref{la:counting-3-factor-sn-2-factor-injective}.
\end{proof}

The finitely many cases not covered by the previous corollary are listed in Table \ref{table:couting-small-numbers}. These numbers were calculated using the Lemmas \ref{la:counting-subdirect-with-2-factor-injective}, \ref{la:counting-2-factor-sn} and \ref{la:counting-3-factor-sn-2-factor-injective}.
However, these numbers were also double-checked with the computer algebra system gap~\cite{GAP4}.

\begin{table}
 \centering
 \begin{tabular}{c||c|c|c|c}
  $n_1 = 4$ & $4$ & $3$ & $2$ & $1$ \\
  \hhline{=#=|=|=|=}
  $4$ & $1386$ & $282$ & $66$ & $32$ \\
  \hline
  $3$ & $282$ & $90$ & $18$ & $8$ \\
  \hline
  $2$ & $66$ & $18$ & $6$ & $2$ \\
  \hline
  $1$ & $32$ & $8$ & $2$ & $1$ \\
 \end{tabular}
 \quad
 \begin{tabular}{c||c|c|c}
  $n_1 = 3$ & $3$ & $2$ & $1$ \\
  \hhline{=#=|=|=}
  $3$ & $90$ & $18$ & $8$ \\
  \hline
  $2$ & $18$ & $6$ & $2$ \\
  \hline
  $1$ & $8$ & $2$ & $1$ \\
 \end{tabular}
 \quad
 \begin{tabular}{c||c|c}
  $n_1 = 2$ & $2$ & $1$ \\
  \hhline{=#=|=}
  $2$ & $6$ & $2$ \\
  \hline
  $1$ & $2$ & $1$ \\
 \end{tabular}
 \caption{The numbers $\ell(n_1,n_2,n_3)$ for $n_1,n_2,n_3 \in \{1,2,3,4\}$}
 \label{table:couting-small-numbers}
\end{table}

\paragraph{Acknowledgments.}
Funded by the Excellence Initiative of the German federal and state governments.

\bibliographystyle{abbrv}
\bibliography{literature}

\end{document}